\documentclass[12pt]{amsart}

\usepackage{amsmath}
\usepackage{amsfonts}
\usepackage{amssymb}
\usepackage{amsthm}
\usepackage{parskip}
\usepackage{enumerate}

\usepackage[colorlinks]{hyperref}

\usepackage{array}
\newcolumntype{x}[1]{>{\centering\arraybfackslash\hspace{0pt}}p{#1}}
\usepackage{multirow}

\usepackage[margin=1in]{geometry}

\usepackage{etoolbox}
\patchcmd{\subsection}{\bfseries}{\itshape}{}{}

\makeatletter
\def\@seccntformat#1{%
  \protect\textup{\protect\@secnumfont
    \ifnum\pdfstrcmp{subsection}{#1}=0 \bfseries\fi
    \csname the#1\endcsname
    \protect\@secnumpunct
  }%
}  
\makeatother
\makeatletter
\def\swappedhead#1#2#3{%
  \thmnumber{\@upn{\the\thm@headfont#2\@ifnotempty{#1}{.~}}}%
  \thmname{#1}%
  \thmnote{ {\the\thm@notefont(#3)}}}
\makeatother

\usepackage{bbm}
\usepackage{parskip}
\usepackage{xcolor}
\usepackage{tikz}
\usepackage{bm}

\numberwithin{equation}{subsection}

\swapnumbers
\newtheorem*{lemma*}{Lemma}
\newtheorem*{proposition*}{Proposition}
\newtheorem*{theorem*}{Theorem}

\newcommand{\ZZ}{\mathbb{Z}}

\DeclareMathOperator{\Lie}{Lie}
\DeclareMathOperator{\Der}{Der}

\DeclareMathOperator{\id}{id}

\DeclareMathOperator{\End}{End}

\DeclareMathOperator{\tr}{tr}

\DeclareMathOperator{\Char}{char}
\DeclareMathOperator{\GL}{GL}

\DeclareMathOperator{\PGO}{PGO}

\DeclareMathOperator{\SL}{SL}
\DeclareMathOperator{\ad}{ad}

\DeclareMathOperator{\Skew}{Skew}

\DeclareMathOperator{\Spin}{\mathbf{Spin}}

\makeatletter
\newcommand{\bigperp}{%
  \mathop{\mathpalette\bigp@rp\relax}%
  \displaylimits
}

\newcommand{\bigp@rp}[2]{%
  \vcenter{
    \m@th\hbox{\scalebox{\ifx#1\displaystyle2.1\else1.5\fi}{$#1\perp$}}
  }%
}
\makeatother

\allowdisplaybreaks

\newtheorem{theorem}[subsection]{Theorem}
\newtheorem{proposition}[subsection]{Proposition}
\newtheorem{lemma}[subsection]{Lemma}
\newtheorem{corollary}[subsection]{Corollary}

\theoremstyle{definition}

\newtheorem*{definition*}{Definition}

\theoremstyle{remark}
\newtheorem{example}[subsection]{Example}

\makeatletter
\def\thm@space@setup{%
	\thm@preskip=\parskip \thm@postskip=6pt
}
\makeatother

\makeatletter
\DeclareFontFamily{OMX}{MnSymbolE}{}
\DeclareSymbolFont{MnLargeSymbols}{OMX}{MnSymbolE}{m}{n}
\SetSymbolFont{MnLargeSymbols}{bold}{OMX}{MnSymbolE}{b}{n}
\DeclareFontShape{OMX}{MnSymbolE}{m}{n}{
    <-6>  MnSymbolE5
   <6-7>  MnSymbolE6
   <7-8>  MnSymbolE7
   <8-9>  MnSymbolE8
   <9-10> MnSymbolE9
  <10-12> MnSymbolE10
  <12->   MnSymbolE12
}{}
\DeclareFontShape{OMX}{MnSymbolE}{b}{n}{
    <-6>  MnSymbolE-Bold5
   <6-7>  MnSymbolE-Bold6
   <7-8>  MnSymbolE-Bold7
   <8-9>  MnSymbolE-Bold8
   <9-10> MnSymbolE-Bold9
  <10-12> MnSymbolE-Bold10
  <12->   MnSymbolE-Bold12
}{}

\let\llangle\@undefined
\let\rrangle\@undefined
\DeclareMathDelimiter{\llangle}{\mathopen}%
                     {MnLargeSymbols}{'164}{MnLargeSymbols}{'164} 

\DeclareMathDelimiter{\rrangle}{\mathclose}%
                     {MnLargeSymbols}{'171}{MnLargeSymbols}{'171}
\makeatother

\title[The Allison--Faulkner construction of $E_8$]{The Allison--Faulkner construction of $E_8$}

\date{\today}

\author{Victor Petrov}
\address{Petrov: St.~Petersburg State University, 29B Line 14th (Vasilyevsky Island), 199178, St.~Petersburg, Russia;
PDMI RAS, Nab. Fontanki 27,
191023, St.~Petersburg, Russia}
\email{victorapetrov@googlemail.com}

\author{Simon W. Rigby}
\address{Rigby: Department of Mathematics: Algebra and Geometry, Ghent University, Krijgslaan 281, 9000 Ghent, Belgium}
\email{simon.rigby@ugent.be}

\thanks{The first author was supported by RFBR grant 19-01-00513. The second author was supported by FWO project G004018N}

\begin{document}

\maketitle

\begin{abstract}
    We show that the Tits index $E_8^{133}$ cannot be obtained by means of the Tits construction over a field with no odd degree extensions. We construct two cohomological invariants, in degrees $6$ and $8$, of the Tits construction and the more symmetric Allison--Faulkner construction of Lie algebras of type $E_8$ and show that these invariants can be used to detect the isotropy rank.
\end{abstract}


\section{Introduction}
Jacques Tits in \cite{tits1966algebres} proposed a general construction of exceptional Lie algebras over an arbitrary field of characteristic not 2 or 3, now called  the Tits construction. The inputs are an alternative algebra and a Jordan algebra, and the result is a simple Lie algebra of type $F_4$, $E_6$, $E_7$ or $E_8$, depending on the dimensions of the algebras. The construction produces, say, all real forms of the exceptional Lie algebras, and a natural question is if all Tits indices can be obtained this way. Skip Garibaldi and Holger Petersson in \cite{garibaldi2016outer} showed that it is not the case for type $E_6$, namely that Lie algebras of Tits index ${}^2E_6^{35}$ do not appear as a result of the Tits construction. We show a similar result for type $E_8$, namely that Lie algebras of Tits index $E_8^{133}$ cannot be obtained by means of the Tits construction, provided that the base field has no odd degree extensions. The proof uses the theory of symmetric spaces and the first author's result with Nikita Semenov and Skip Garibaldi about isotropy of groups of type $E_7$ in terms of the Rost invariant \cite{garibaldi2015shells}.

We prefer to use a more symmetric version of Tits construction due to Bruce Allison and John Faulkner. Here the input is a so-called structurable algebra with an involution (say, the tensor product of two octonion algebras) and three constants. The Lie algebra is given by some Chevalley-like relations. The Tits construction and the Allison--Faulkner construction have a large overlap but, strictly speaking, neither one is more general than the other. The Tits construction is capable of producing Lie algebras of type $E_8$ whose Rost invariant has a non-zero 3-torsion part (necessarily using a Jordan division algebra as input), but the Allison--Faulkner construction of $E_8$ cannot do this -- at least when the input is a form of the tensor product of two octonion algebras, because these can  always be split by a 2-extension of the base field. On the other hand, the Allison--Faulkner construction is capable of producing Lie algebras of type $E_8$ with the property that the 2-torsion part of their Rost invariant has symbol length 3, and this is impossible for the Tits construction (see \cite[11.6]{garibaldi2007orthogonal}). An $E_8$ with this property  would necessarily come from what we call an indecomposable bi-octonion algebra, and these are related to some unusual examples of 14-dimensional quadratic forms discovered by Oleg Izhboldin and Nikita Karpenko \cite{izhboldin2000some}.

We produce two new cohomological invariants, one in degree $6$ and one in degree~$8$, and show that these invariants can be used to detect the isotropy rank of either the Tits or the Allison--Faulkner construction (but unlike the results of \cite{garibaldi2016outer} we give necessary conditions only). The main tool for constructing these invariants is a calculation of the Killing form of an Allison--Faulkner construction which, under a mild condition on the base field, is  near to an $8$-Pfister form (a so-called Pfister neighbour).

\section{Preliminaries}

Let $K$ be a field of characteristic not 2 or 3. If $q$ is a quadratic form, we write $q(x,y) = q(x+y)-q(x)-q(y)$. If $A$ is an algebra and $a \in A$, we denote by $L_a, R_a \in \End(A)$ the left- and right-multiplication operators, respectively.

\subsection{Bi-octonion algebras}

	
A $K$-algebra with involution $(A,-)$ is called a \emph{decomposable bi-octonion algebra} if it has two octonion subalgebras $C_1$ and $C_2$ that are stabilised by the involution, such that $A = C_1 \otimes_K C_2$. A \emph{bi-octonion algebra} is an algebra with involution $(A,-)$ that becomes isomorphic to a decomposable bi-octonion algebra over some field extension. These are important examples of central simple structurable algebras, as defined by Allison in \cite{allison1978class}, and they are instrumental in constructing Lie algebras of type $E_8$ (see \ref{sec:AF construction}).

Any bi-octonion algebra $(A,-)$ is either decomposable or it decomposes over a unique quadratic field extension $E/K$. In the latter case, there exists an octonion algebra  $C$ over~$E$, unique up to $k$-isomorphism, from which $(A,-)$ can be reconstructed as follows. Let $\iota$ be the non-identity automorphism of $E/K$, and let ${}^\iota C$ be a copy of $C$ as a $k$-algebra, but with a different $E$-algebra structure given by $e \cdot z = \iota(e)z$. Then $(A,-)$ is precisely the fixed point set of ${}^\iota C \otimes_E C$ under the $k$-automorphism $x \otimes y \mapsto y \otimes x$, with the involution being the restriction of the tensor product of the canonical involutions on $^\iota C$ and $C$  \cite[Theorem 2.1]{allison1988tensor}. We denote this algebra by $(A,-) = N_{E/K}(C)$.

To unify the description of both decomposable and non-decomposable bi-octonion algebras, if we consider $C=C_1 \times C_2$ as an octonion algebra over the split quadratic \'etale extension $K\times K$, then $N_{K\times K/K}(C)$ as defined above is just isomorphic to $C_1\otimes_K C_2$.

\subsection{Multiplicative transfer of quadratic forms}

Markus Rost defined a multiplicative analogue of the Scharlau transfer for quadratic forms, and it has been studied by him and his students (e.g., in \cite{rost2002pfister, wittkop}) and used before to define cohomological invariants.

If $(q,V)$ is a quadratic space over a quadratic \'etale extension $E/K$,  one defines the quadratic space $({}^\iota q,{}^\iota V)$ where $\iota$ is the non-trivial automorphism of $E/K$, $^{\iota} V$ is a copy of $V$ as a $k$-vector space but with the action of $E$ modified by $\iota$, and  $^\iota q (v) = \iota(q(v))$. Then the \emph{multiplicative transfer} $N_{E/K}(q)$ of $q$ is the quadratic form obtained by restricting ${^\iota}q\otimes_E q$ to the $k$-subspace of  tensors in ${^\iota V} \otimes_E V$ fixed by the switch map $x \otimes y \mapsto y \otimes x$.

In the case of a split extension,  a quadratic form over $K\times K$ is just a pair $(q_1, q_2)$ where $q_1, q_2$ are quadratic forms over $K$ of the same dimension, and $N_{K\times K/K}(q_1, q_2) = q_1 \otimes q_2$.

\begin{lemma}
    Let $(A,-) = N_{E/K}(C)$ for an octonion algebra $C$ over a quadratic \'etale extension $E/K$, and let $n$ be the norm of $C$. Then $N_{E/K}(n)$ equals the normalised trace form $(x,y) \mapsto \frac{1}{128}\tr(L_{x \bar y + y \bar x})$.
\end{lemma}

\begin{proof}
    Both $N_{E/K}(n)$ and the normalised trace form are invariant symmetric bilinear forms on $(A,-)$ in the sense that Allison defined (see \cite[Theorem 17]{allison1978class} and \cite[Proposition 2.2]{allison1988tensor}). By a theorem of Schafer \cite{schafer1989invariant}, a central simple structurable algebra has at most one such bilinear form, up to a scalar multiple. (As discussed in \cite[p.\!~116--117]{schafer1989invariant}, these facts are valid in characteristic 0 or $p \ge 5$, despite some of the original references being limited to characteristic~0.)
\end{proof}

\subsection{Lie-related triples}
Let $(A,-)$ be a central simple structurable algebra over $K$.
A \emph{Lie related triple} (in the sense of \cite[\S3]{allison1993nonassociative}) is a triple $T = (T_1, T_2, T_3)$ where $T_i \in \End(A)$ and 
\[
\overline{T_i\big(\ \overline{xy}\ \big)}  = T_j(x)y + xT_k(t)
\]
for all $x, y \in A$ and all $(i\ j\ k)$ that are cyclic permutations of $(1\ 2\ 3)$. Define $\mathcal{T}$ to be the Lie subalgebra of $\mathfrak{gl}(A) \times \mathfrak{gl}(A) \times \mathfrak{gl}(A)$ spanned by the set of related triples.

For $a, b \in A$ and $1 \le i \le 3$, define
\[
T_{a,b}^{i} = (T_1, T_2, T_3)
\]
where (taking indices mod 3):
\begin{align*}
T_i	& = L_{\bar b} L_a - L_{\bar a } L_b,\\
T_{i+1} &= R_{\bar b} R_a - R_{\bar a} R_b, \\
T_{i+2} &= R_{\bar a  b - \bar b  a} + L_{b}L_{\bar a} - L_a L_{\bar b}.
\end{align*}

Let $\mathcal{T}_I$ be the subspace of $\End(A)^3$ spanned by $\{T_{a,b}^i \mid a, b \in A, 1 \le i \le 3\}$. Since $(A,-)$ is structurable, $\mathcal{T}_I$ is a Lie subalgebra of $\mathcal{T}$ \cite[Lemma 5.4]{allison1993nonassociative}. 
Finally, denote by $\Skew(A,-)\subset A$  the $(-1)$-eigenspace of the involution, and let $\mathcal{T}'$ be the subspace of $\End(A)^3$ spanned by triples of the form
	\begin{equation} \label{eq:decomposition1}
	(D,D,D) + (L_{s_2}-R_{s_3}, L_{s_3} -R_{s_1}, L_{s_1} - R_{s_2})
	\end{equation}
	where $D \in \Der(A,-)$ and $s_i \in \Skew(A,-)$ with $s_1 + s_2 + s_3 = 0$. 

\begin{example}
    Let $(C,-)$ be an octonion algebra with norm $n$. The principle of local triality holds in $\mathcal{T}_I$ in the sense that each of the projections $\mathcal{T}_I \to \mathfrak{gl}(C)$, $(T_1, T_2, T_3) \mapsto T_i$, for $1 \le i \le 3$, is injective \cite[Theorem~3.5.5]{springer-veldkamp}. The Lie algebra $\mathcal{T}_I$ is isomorphic to $\mathfrak{so}(n)$ \cite[Lemma~3.5.2]{springer-veldkamp}.
    The $(i+2)$-th entry of the triple $T_{a,b}^i$ is $R_{\bar a  b - \bar b  a} + L_{b}L_{\bar a} - L_a L_{\bar b}$, and by \cite[pp.\!~51,~54]{springer-veldkamp} this is the map $C \to C$ that sends
    \begin{equation} \label{tab}
        x \mapsto 2n(x,a)b - 2n(x,b)a.
    \end{equation}
\end{example}

\begin{proposition} \label{prop:T_I for bioct}
	If $(A,-)$ is a bi-octonion algebra of the form $(A,-) = N_{E/K}(C)$ for some quadratic \'etale extension $E/K$ and some octonion algebra $C$ over $E$, then $\mathcal{T}_I = \mathcal{T}_0 = \mathcal{T}' \simeq \Lie(R_{E/k}(\mathbf{Spin}(n)))$, where $n$ is the norm of $C$.
 \end{proposition}
\begin{proof}
	We have that $\mathcal{T}_I \subset \mathcal{T} \subset \mathcal{T}'$  and $\dim \mathcal{T}' = \dim \Der(A,-) + 2 \dim\Skew(A,-) = 28 + 28 = 56$ by \cite[Corollary~3.5]{allison1993nonassociative}. On the other hand, $\mathcal{T}_I$ (as an $E$-module) is precisely $\Lie(\mathbf{Spin}(n))$ \cite[Theorem 3.5.5]{springer-veldkamp} and so $\mathcal{T}_I$ (as a $K$-vector space) is 56-dimensional and isomorphic to $\Lie(R_{E/K}(\mathbf{Spin}(n)))$.
\end{proof}

\subsection{Local triality} \label{sec:local-triality}

In the context of Proposition \ref{prop:T_I for bioct}, the Lie algebra $\mathcal{T}_I$ is of type $D_4 + D_4$. Local triality holds here too: the projections $\mathcal{T}_I \to \mathfrak{gl}(A), (T_1, T_2,T_3) \mapsto T_i$ are injective for any $1 \le i \le 3$, and the symmetric group $S_3$ acts on $\mathcal{T}_I$ by $E$-automorphisms, where $E$ is the centroid of $\mathcal{T}_I$ (compare with \cite[\S3.5]{springer-veldkamp}).

\subsection{The Allison--Faulkner construction \cite[\S4]{allison1993nonassociative}} \label{sec:AF construction}
Let $(A,-)$ be a central simple structurable algebra and let $\gamma = (\gamma_1, \gamma_2, \gamma_3) \in K^{\times}\times K^\times \times K^\times$. For $1 \le i,j \le 3$ and $i \ne j$, define $A[ij] = \{a[ij] \mid a \in A \}$ to be a copy of $A$, and identify $A[ij]$ with $A[ji]$ by setting $a[ij] = -\gamma_i \gamma_j^{-1} \overline {a}[ji]$. Define the vector space
\[
K(A,-,\gamma) = \mathcal{T}_I \oplus A[12] \oplus A[23] \oplus A[31]
\]
	and equip it with an algebra structure defined by the multiplication:
\begin{align*}
	[a[ij], b[jk]] &= -[b[jk],a[ij]] = ab[ik],\\
	[T, a[ij]] &= - [a[ij],T] = T_k(a)[ij]\\
	[a[ij],b[ij]] &= \gamma_i \gamma_j^{-1} T^i_{a,b}
\end{align*}
	for all $a, b \in A$, $T = (T_1, T_2, T_3) \in \mathcal{T}_I$, and $(i\ j\ k)$ a cyclic permutation of $(1\ 2\ 3)$. 	Then $K(A,-,\gamma)$ is clearly a $\ZZ/2\ZZ\times \ZZ/2\ZZ$-graded algebra, and it is in fact a central simple Lie algebra \cite[Theorems~4.1, 4.3, 4.4, \&~5.5]{allison1993nonassociative}. 

\subsection{Relation to the Tits--Kantor--Koecher construction} 
  If the quadratic form $\langle \gamma_1,\gamma_2,\gamma_3 \rangle$ is isotropic then $K(A,-,\gamma) \simeq K(A,-)$ where \[K(A,-) = \Skew(A) \oplus A \oplus V_{A,A} \oplus A \oplus \Skew(A)\] is the  Tits--Kantor--Koecher construction \cite[Corollary~4.7]{allison1991construction}. An isomorphism and its inverse are determined explicitly in \cite[Theorem 2.2]{allison1991construction} in the case where $-\gamma_1\gamma_2^{-1}$ is a square. More generally,   if  $\langle \gamma_1,\gamma_2,\gamma_3 \rangle$  and $\langle \gamma_1',\gamma_2',\gamma_3'\rangle$ are similar quadratic forms, then $K(A,-,\gamma) \simeq K(A,-,\gamma')$ \cite[Proposition~4.1]{allison1991construction}.
  
  In particular, if $(A,-)$ is a bi-octonion algebra, then $K(A,-,\gamma)$ is a central simple Lie algebra of type $E_8$.
	
\subsection{Relation to the Tits construction} \label{sec: relation to tits construction}
    Jacques Tits in \cite{tits1966algebres} defined the following construction of Lie algebras. Let $C$ be an alternative algebra and $J$ be a Jordan algebra. Denote by $C^\circ$ and $J^\circ$ the subspaces of elements of generic trace zero and define operations $\circ$ and bilinear forms $(-,-)$ on $C^\circ$ and $J^\circ$ by the formula
    $$
    ab=a\circ b+(a,b)1.
    $$
    Two elements $a,b$ in $J$ and $C$ define an inner derivation $\langle a,b\rangle$ of the respective algebra, namely:
    \[\langle a,b \rangle (x) = \tfrac{1}{4}[[a,b],x]- \tfrac{3}{4}[a,b,x]. \]
    Then there is a Lie algebra structure on the vector space $\Der(J)\oplus J^\circ\otimes C^\circ\oplus\Der(C)$ defined by the formulas
    \begin{align*}
    &[\Der(J),\Der(C)]=0;\\
    &[B+D,a\otimes c]=B(a)\otimes c+a\otimes D(c);\\
    &[a\otimes c,a'\otimes c']=(c,c')\langle a,a'\rangle+(a\circ a')\otimes(c\circ c')+(a,a')\langle c,c'\rangle
    \end{align*}
    for all $B \in \Der(J)$, $D \in \Der(C)$, $a,a' \in J^\circ$, and $c,c' \in C^\circ$.
	If $(A,-) = C_1 \otimes C_2$ is a decomposable bi-octonion algebra, then 
	$K(A,-,\gamma)$ is isomorphic to the Lie algebra obtained via the Tits construction from the composition algebra $C_1$ and the reduced Albert algebra $\mathcal{H}_3(C_2,\gamma)$ \cite[Remark~1.9~(c)]{allison1991construction}.
	
\begin{proposition}\label{prop:D8}
    Let $(A,-) = C_1 \otimes C_2$ be a decomposable bi-octonion algebra. Then \[\mathcal{T}_I \oplus A[ij] \simeq \mathfrak{so}(\langle \gamma_i\rangle n_1 \perp \langle -\gamma_j^{-1} \rangle n_2),\] where $n_\ell$ is the norm of $C_\ell$.
\end{proposition}

\begin{proof}

    Consider the quadratic form $Q  = \langle \gamma_i \rangle n_1 \perp \langle - \gamma_j^{-1} \rangle n_2$  on the vector space $C_1 \oplus C_2$. The Lie algebra $\mathfrak{so}(Q)$ can be embedded into the  Clifford algebra $C(Q)$ as the subspace spanned by elements of the form
    \begin{align*}
        [u,v]_{c}, && u, v \in C_1 \oplus C_2
    \end{align*}
    where $[-,-]_c$ denotes the commutator in the Clifford algebra (to avoid confusion with the commutators in $C_1$ and $C_2$). These generators satisfy the relations \cite[p.\!~232~(30)]{jacobson1979lie}:
    \begin{align} \label{eq:relations}
        [[u,v]_c,[u',v']_c]_c = -2Q(u,u')[v,v']_c + 2Q(u,v')[v,u']_c + 2Q(v,u')[u,v']_c-2Q(v,v')[u,u']_c.
    \end{align}

    If $z,z' \in C_1$ and $w,w' \in C_2$, this becomes
    \begin{align} \label{eq:commutator zw}
        [[z,w]_c,[z',w']_c]_c = -2\gamma_in_1(z,z')[w,w']_c + 2\gamma_j^{-1}n_2(w,w')[z,z']_c.
    \end{align}
    This implies that the 64-dimensional subspace spanned by
    \begin{align*}
        [z,w]_c, && z \in C_1, w \in C_2
    \end{align*}
    generates the Lie algebra $\mathfrak{so}(Q)$. Now define a linear bijection $\theta: \mathfrak{so}(Q) \to \mathcal{T}_I \oplus A[ij]$ by \begin{align*}
        [z,z']_c &\mapsto \gamma_iT^i_{z,z'},\\
        [w,w']_c &\mapsto -\gamma_j^{-1}T^i_{w,w'},\\
        [z,w]_c &\mapsto zw[ij]
    \end{align*}
    for all $z,z' \in C_1$ and $w, w' \in C_2$. By \cite[p.\!~232~(31)]{jacobson1979lie} and \eqref{tab}, the restriction of $\theta$ to the subalgebra $[C_1,C_1]_c \oplus [C_2,C_2]_c\simeq \mathfrak{so}(\langle \gamma_i \rangle n_1)\times \mathfrak{so}(\langle -\gamma_j^{-1} \rangle n_2)$ is a homomorphism.

   Now we calculate using \eqref{eq:commutator zw} that
    \begin{align}
        \theta([[z,w]_c,[z',w']_c]_c)  &= \theta(-2\gamma_i n_1(z,z')[w,w']_c+2\gamma_j^{-1}
        n_2(w,w')[z,z']_c) \notag \\
        &= -2\gamma_i n_1(z,z') \theta([w,w']_c)+2\gamma_j^{-1} n_2(w,w')\theta([z,z']_c)  \notag \\ 
        &= 2\gamma_i\gamma_j^{-1} \big(n_1(z,z') T_{w, w'}^i + n_2(w,w')T^i_{z,z'}\big). \label{eq:theta-hom1}
    \end{align}
    Meanwhile, we have 
    \begin{align}
        [\theta([z,w]_c),\theta([z',w']_c)] = [zw[ij],z'w'[ij]] = \gamma_i\gamma_j^{-1} T_{zw,z'w'}^i \label{eq:theta-hom2}
    \end{align}
    To complete the proof that $\theta$ is an isomorphism, we show that the triples \eqref{eq:theta-hom1} and \eqref{eq:theta-hom2} are equal. It suffices to compare the $i$-th entries of each triple (by \S\ref{sec:local-triality}). After recalling that \[L_{x}L_{\overline{x'}} + L_{x'}L_{\overline{x}} = L_{\overline x} L_{x'} + L_{\overline{ x'}}L_x = n_i(x,x') \id\] for all $x \in C_\ell$ \cite[Lemma~1.3.3~(iii)]{springer-veldkamp}, the $i$-th entry of \eqref{eq:theta-hom1} is
    \begin{align*}
        &~2\gamma_i\gamma_j^{-1}\big(n_1(z,z')(L_{\overline{w'}}L_w - L_{\overline{w}}L_{w'})+n_2(w,w')(L_{\overline{z'}}L_z - L_{\overline{z}}L_{z'})\big) \\
        = &~2\gamma_i\gamma_j^{-1}\big((L_{\overline{z}}L_{z'} + L_{\overline{z'}}L_{{z}})(L_{\overline{w'}}L_w - L_{\overline{w}}L_{w'}) + (L_{\overline{w}}L_{w'} + L_{\overline{w'}}L_{w}) (L_{\overline{z'}}L_z - L_{\overline{z}}L_{z'})\big)\\
        =&~2\gamma_i\gamma_j^{-1}(2L_{\overline{z'}}L_zL_{\overline{w'}}L_w - 2 L_{\overline{z}}L_{z'}L_{\overline{w}}L_{w'}) =4 \gamma_i\gamma_j^{-1}(L_{\overline{z'}}L_z L_{\overline{w'}} L_w - L_{\overline{z}}L_{z'}L_{\overline{w}} L_{w'})
    \end{align*}
    In the last  line, we have used (multiple times) the fact that $C_1$ and $C_2$ commute and associate with each other. Using this fact a few more times, the $i$-th entry of \eqref{eq:theta-hom2} is just
     \begin{align*}
         4\gamma_i\gamma_j^{-1}(L_{\overline{z'w'}}L_{zw} - L_{\overline{zw}}L_{z'w'}) &= 4\gamma_i\gamma_j^{-1} (L_{\overline{z'}}L_z L_{\overline{w'}}L_w  - L_{\overline{z}}L_{z'}L_{\overline{w}}L_{w'}). \qedhere
     \end{align*}
    \end{proof}

\section{The Killing form of $K(A,-,\gamma)$}
	For any quadratic form $q = \langle x_1, \dots, x_n\rangle$, the Killing form of $\mathfrak{so}(q)$ is 
	\begin{equation} \label{killing form so(q)}
	\langle 4-2n \rangle \lambda^2(q)
	\end{equation}
	where $\lambda^2(q) = \bigperp_{i < j}{\langle x_i x_j\rangle }$ \cite[Exercise~19.2]{garibaldi2009cohomological}.
	
	\begin{lemma} \label{lem:invariant_form}
		Let $A  = N_{E/K}(C)$ as before, and let $\rho_{ij}: R_{E/K}(\mathbf{Spin}(n)) \to \mathbf{GL}(A[ij])$ be the representation lifted from the representation of $\mathcal{T}_I$ in $A[ij]$. Every quadratic form $q$ on $A$ invariant under this action of $R_{E/K}(\mathbf{Spin}(n))$ is a scalar multiple of the multiplicative transfer $N_{E/K}(n)$ (equivalently, a scalar multiple of the trace form $(x,y)\mapsto \tr(L_{x \bar y +y \bar x})$).
	\end{lemma}
	
	\begin{proof}
		We can extend scalars from $K$ to $E$, and then $q_E$ is a quadratic form on $A_E = C \otimes_E {^\iota C}$ which is invariant under the action of $R_{E/K}(\mathbf{Spin}(n))\times_K E = \mathbf{Spin}(n) \times \mathbf{Spin}(^{\iota} n)$. Then clearly $q_E$ decomposes as $q_1 \otimes q_2$ for some $\mathbf{Spin}(n)$-invariant form $q_1$ on $C$ and some $\mathbf{Spin}({^\iota n})$-invariant form  $q_2$ on ${^\iota C}$. This implies $q_1 \simeq \langle \lambda_1 \rangle n$ and $q_2 \simeq \langle \lambda_2 \rangle {^\iota n}$ for certain scalars $\lambda_i \in E^\times$, and therefore $q_E = \langle \lambda_1\lambda_2 \rangle n \otimes {^\iota n}$. However, since $(q_E, A_E)$ is extended from $(q,A)$ and $n\otimes{^\iota n}(1\otimes 1) =1$, we have $\lambda_1 \lambda_2 \in K^\times$. Therefore, $q = q_E|_A = \langle \lambda_1 \lambda_2 \rangle N_{E/K}(n)$.
	\end{proof}
	
	We can now calculate the Killing form of $K(A,-,\gamma)$ in the case where $(A,-)$ is a bi-octonion algebra.
	
	\begin{proposition} \label{prop: killing form calculation}
	If $(A,-)=N_{E/K}(C)$, then the Killing form on $K(A,-,\gamma)$ is 
	\begin{equation} \label{eq:killing-form}
	\langle-15\rangle\big(\tr_{E/K}(\lambda^2(n)) \perp\langle \gamma_1\gamma_2^{-1},\gamma_2\gamma_3^{-1},\gamma_3\gamma_1^{-1} \rangle N_{E/K}(n)\big).
	\end{equation}
	\end{proposition}

	\begin{proof}
	Let $\kappa$ be the Killing form of $K(A,-,\gamma)$. If $x, y \in K(A,-,\gamma)$ are from different homogeneous components in the $\ZZ/2\ZZ\times \ZZ/2\ZZ$-grading, then $\ad_x \ad_y$ shifts the grading and consequently $\kappa(x,y) = \tr(\ad_x \ad_y) = 0$. 
	
	Let $\tau$ be the Killing form of $\mathcal{T}_I$. The Killing form of $\Lie(\mathbf{Spin}(n))$ is $\langle -3 \rangle \lambda^2(n)$; see \eqref{killing form so(q)}. Since $\mathcal{T}_I \simeq \Lie(R_{E/K}(\mathbf{Spin}(n))$ by Proposition~\ref{prop:T_I for bioct},   we have $\tau = \tr_{E/K}(\langle -3 \rangle \lambda^2(n)) = \langle-3\rangle \tr_{E/K}(\lambda^2(n))$. There is an automorphism of $K(A,-,\gamma)\otimes_K K^{\rm alg}$ that swaps the two simple subalgebras of $\mathcal{T}_I\otimes_K K^{\rm alg}$, and this implies  $\kappa|_{\mathcal{T}_I}$ is a scalar multiple of $\tau$; say $\kappa|_{\mathcal{T}_I} = \langle \phi_0\rangle\langle -3 \rangle\tr_{E/K}(\lambda^2(n))$ for some $\phi_0\in K^\times$.
	
	Let us determine $\phi_0$. The grading on $K(A,-,\gamma)$ makes it a sum of four $\mathcal{T}_I$-modules. For $T, S \in \mathcal{T}_I$ and $a \in A$,
	\[
	[T,[S,a[ij]]] = T_k(S_k(a))[ij].
	\]
	Therefore \[\kappa(T,S) = \tr(\ad_T \ad_S) = \tau(T,S)+ \tr(T_1S_1) + \tr(T_2S_2)+\tr(T_3S_3).\]
	The trace forms of the irreducible representations $\mathcal{T}_I \to \mathfrak{gl}(A)$, $T \mapsto T_\ell$ for $1 \le \ell \le 3$ are all equal (despite them being inequivalent representations) and so  $\tr(T_1S_1) = \tr(T_2S_2) = \tr(T_3S_3)$ for all $T,S \in \mathcal{T}_I$. Moreover, $\tr(T_1S_1)$ is a scalar (in fact, integer) multiple of $\tau(T,S)$.
	
		To determine the ratio between $\tr(T_1 S_1)$ and $\tau(T,S)$,  we can assume $A = C_1 \otimes C_2$ is decomposable, and consider the subalgebra $\mathfrak{so}(n_1)\subset \mathfrak{so}(n_1)\times \mathfrak{so}(n_2) \simeq \Lie(R_{E/K}(\Spin(n))$, where $n_\ell$ is the norm on $C_\ell$. It is well-known that the Killing form $\kappa_1$ on $\mathfrak{so}(n_1)$ is $6\ (=8-2)$ times the trace form of its vector representation $\mathfrak{so}(n_1)\to \mathfrak{gl}(C_1)$, while the trace form of the representation $\mathfrak{so}(n_1) \to \mathfrak{gl}(C_1\otimes C_2)$ is clearly 8 times the trace form of the vector representation. But 
	 $\kappa_1$ is equal to the restriction of the Killing form $\tau$ on $\mathfrak{so}(n_1)\times \mathfrak{so}(n_2)$, so this means that (if $T \in \mathcal{T}_I$ belongs to the $\mathfrak{so}(n_1)$ subalgebra) we have $\tr(T_1^2) = 8\tr({T_1|_{C_1}}^2) = \frac{8}{6}\kappa_1(T) = \frac{8}{6}\tau(T)$. In conclusion, $\phi_0 =5$, so $\kappa|_{\mathcal{T}_I} = \langle -15\rangle \tr_{E/K}(\lambda^2(n))$.

	The restriction $\kappa|_{A[ij]}$ is an invariant form under the action of $R_{E/K}(\mathbf{Spin}(n))$, which means it is proportional to $N_{E/K}(n)$, by Lemma~\ref{lem:invariant_form}. Say $\kappa|_{A[ij]} = \langle\phi_{ij}\rangle N_{E/K}(n)$.
	To determine the $\phi_{ij}$, it suffices to calculate $\kappa(1[ij])$, since $\kappa(1[ij]) = \phi_{ij}N_{E/K}(n)(1) = \phi_{ij}$.  By definition $\kappa(1[ij])$ is the trace of ${\ad_{1[ij]}}^2$. The graded components of $K(A,-,\gamma)$ are invariant under ${\ad_{1[ij]}}^2$, so we work out the trace separately for each of these components.
	
	For all $b \in A$, we have
	\begin{align*}
	[1[ij],[1[ij],b[jk]]] &= [1[ij],b[ik]] = -\gamma_i\gamma_j^{-1}[1[ji],b[ik]] = -\gamma_i\gamma_j^{-1}b[jk]
	\end{align*}
	so ${\ad_{1[ij]}}^2|_{A[jk]} = -\gamma_i\gamma_j^{-1}\id$, and $\tr({\ad_{1[ij]}}^2|_{A[jk]}) = -64\gamma_i\gamma_j^{-1}$.
	Similarly, for all $b \in A$, 
	\begin{align*}
		[1[ij],[1[ij],b[ki]]] &= (-\gamma_i\gamma_j^{-1})(-\gamma_k\gamma_i^{-1})[1[ij],[1[ji],\bar{b}[ik]]] = (-\gamma_i\gamma_j^{-1})(-\gamma_k\gamma_i^{-1})[1[ij],\bar{b}[jk]] \\ &= (-\gamma_i\gamma_j^{-1})(-\gamma_k\gamma_i^{-1})[\bar{b}[ik]] = (-\gamma_i\gamma_j^{-1})(-\gamma_k\gamma_i^{-1})(-\gamma_i\gamma_k^{-1})b[ki]\\ &= -\gamma_i\gamma_j^{-1}b[ki],
	\end{align*}
	so ${\ad_{1[ij]}}^2|_{A[ki]} = -\gamma_i \gamma_j^{-1}\id$, and $\tr({\ad_{1[ij]}}^2|_{A[jk]}) = -64\gamma_i\gamma_j^{-1}$. 
	In contrast, for all $b \in A$, 
	\begin{align*}
	[1[ij],[1[ij],b[ij]]] &= [1[ij],\gamma_i\gamma_j^{-1}T_{1,b}^i] = -\gamma_i\gamma_j^{-1}(T_{1,b}^i)_k(1)\\
	&= -\gamma_i\gamma_j^{-1}(R_{b-\bar b}+L_b-L_{\bar b})(1) = -2\gamma_i\gamma_j^{-1}(b-\bar b).
	\end{align*}
	Therefore, ${\ad_{1[ij]}}^2|_{A[ij]}$ has a 50-dimensional kernel $\{a[ij] \mid \bar a =a \}$ and a 14-dimensional eigenspace $\{a[ij] \mid \bar a = -a \}$ with eigenvalue $-4\gamma_i\gamma_j^{-1}$. This proves that $\tr({\ad_{1[ij]}}^2|_{A[ij]}) = -56\gamma_i\gamma_j^{-1}$.
	
    Now if $T = (T_1, T_2, T_3) \in \mathcal{T}_I$, then
    \[
    [1[ij], [1[ij],T]] = [1[ij],-T_k(1)[ij]] = -\gamma_i\gamma_j^{-1} T^i_{1,T_k(1)}.
    \]
    We can use \eqref{eq:decomposition1} to write $T = (D,D,D) + (L_{s_2}-R_{s_3}, L_{s_3}-R_{s_1}, L_{s_1} - R_{s_2})$ for some unique $D \in \Der(A,-)$ and $s_i \in \Skew(A,-)$ such that $s_1 + s_2 + s_3 = 0$. Note that the $k$-th entry of $T$ is $L_{s_i}-R_{s_j}$. Then $T_k(1) = D(1) + L_{s_i}(1)-R_{s_j}(1) = s_i - s_j$, so  ${\ad_{1[ij]}}^2(T) = -\gamma_i\gamma_j^{-1}T^i_{1,T_k(1)}$ is the triple whose $k$-th entry is
    \begin{align*}
    -\gamma_i\gamma_j^{-1} (R_{1T_k(1) - \overline{T_k(1)}1} + L_{T_k(1)}L_1 - L_1L_{\overline{T_k(1)}}) &= -\gamma_i\gamma_j^{-1}(R_{2(s_i-s_j)}+L_{2(s_i-s_j)}) \\&= -2\gamma_i\gamma_j^{-1}\big((L_{s_i}-R_{s_j}) - (L_{s_j}-R_{s_i})\big).
    \end{align*}
     This shows $\ker({\ad_{1[ij]}}^2|_{\mathcal{T}_I})$ is the 42-dimensional subspace of $\mathcal{T}_I$ whose $k$-th projection is
    \[
    \{D + L_s - R_s \mid D \in \Der(A,-), s \in \Skew(A,-)\}.
    \]
    And the subspace of $\mathcal{T}_I$ whose $k$-th projection is
    \[
    \{L_s + R_s \mid s \in \Skew(A,-)\}
    \]
    is a 14-dimensional eigenspace of ${\ad_{1[ij]}}^2|_{\mathcal{T}_I}$ with eigenvalue   $-4\gamma_i\gamma_j^{-1}$.
    This proves that $\tr({\ad_{1[ij]}}^2|_{\mathcal{T}_I}) = -56\gamma_i\gamma_j^{-1}$. 	Therefore
	\[
	 \phi_{ij}  = \kappa(1[ij]) = \tr({\ad_{1[ij]}}^2) = -64\gamma_i\gamma_j^{-1} - 64\gamma_i\gamma_j^{-1}  -56\gamma_i\gamma_j^{-1}  -56\gamma_i\gamma_j^{-1} = -240\gamma_i\gamma_j^{-1},
	\]
	and we can simplify to get \eqref{eq:killing-form} because 15 is in the same square class as 240.
	\end{proof}

If $\Char(K) = 5$, then the form Killing form on $E_8$ is zero, but deleting the factor $\langle -15 \rangle$ in \eqref{eq:killing-form} gives a nondegenerate quadratic form on $K(A,-,\gamma)$ such that the associated bilinear form is Lie invariant (see \cite[p.\! 117]{jacobson1971exceptional}).

Assume for the rest of this section that $-1$ is a sum of two squares in $K$; equivalently, $4 = 0$ in the Witt ring $W(K)$. 

\begin{lemma}
Let $(A,-) = N_{E/K}(C)$, and let $\kappa'$ be a nondegenerate Lie invariant bilinear form on $K(A,-,\gamma)$. Then $\kappa' \in I^6(K)$ and there is a unique $64$-dimensional form $q \in I^6(K)$ such that $q+ \kappa'  \in I^8(K)$.
\end{lemma}

\begin{proof}
Since $\kappa'$ is unique up to a scalar multiple, we can assume without loss of generality that \[\kappa' = \big(\tr_{E/K}(\lambda^2(n)) \perp\langle \gamma_1\gamma_2^{-1},\gamma_2\gamma_3^{-1},\gamma_3\gamma_1^{-1} \rangle N_{E/K}(n)\big).\] Then $\kappa' \in I^6(K)$, since $\tr_{E/K}(\lambda^2(n)) = 0$ \cite[Lemma 19.8]{garibaldi2009cohomological} and $N_{E/K}(n) \in I^6(K)$ \cite{rost2002pfister}, \cite[Satz 2.16]{wittkop}. Setting $q = N_{E/K}(n)$ yields
\[
q + \kappa' = \langle 1, \gamma_1\gamma_2^{-1}, \gamma_2\gamma_3^{-1}, \gamma_3\gamma_1^{-1} \rangle N_{E/K}(n) = \llangle -\gamma_1\gamma_2^{-1}, -\gamma_2\gamma_3^{-1} \rrangle N_{E/K}(n) \in I^8(K).
\]
The uniqueness follows from the Arason--Pfister Hauptsatz.
\end{proof}

Let $Q(*) \subset R(*) \subset  H^1(*,E_8)$ be the functors $\mathsf{Fields}_{/K} \to \mathsf{Sets}$ such that for all fields $F/K$:
\begin{enumerate}[\rm (i)]
\item $Q(F)$ is the set of isomorphism classes of Lie algebras of type $E_8$ that are isomorphic to $K(A,-,\gamma)$ for some bi-octonion algebra $(A,-)$ over $F$ and some $\gamma = (\gamma_1, \gamma_2, \gamma_3) \in (K^\times)^3$; i.e.\! $Q(F)$ is the image of the Allison--Faulkner construction \[H^1(F,(G_2\times G_2\rtimes \ZZ/2\ZZ)\times (\ZZ/2\ZZ)^3) \to H^1(F,E_8).\] \item $R(F)$ is the set of isomorphism classes of Lie algebras $L$ of type $E_8$ such that the class of $L$ is contained in $Q(F')$ for some odd-degree extension $F'/F$.
\end{enumerate}

Recall from \ref{sec: relation to tits construction} that $Q(*)$ contains all Lie algebras of type $E_8$ that are obtainable using the Tits construction from a reduced Albert algebra and an octonion algebra. Whereas, $R(*)$ strictly contains all Lie algebras of type $E_8$ that are obtainable using the Tits construction from an Albert algebra (even a division algebra) and an octonion algebra. Any cohomological invariant $Q(*) \to \bigoplus_{i \ge 0} H^i(*,\ZZ/2\ZZ)$ can be extended uniquely to a cohomological invariant $R(*) \to \bigoplus_{i \ge 0} H^i(*,\ZZ/2\ZZ)$ \cite[\S7]{garibaldi2009cohomological}.

By applying the quadratic form invariants $e_n: I^n(*) \to H^n(*,\ZZ/2\ZZ)$ for $n = 6$ and $8$, we obtain cohomological invariants of the Tits construction and the Allison--Faulkner construction.

\begin{corollary} \label{existence of h8}
 If $-1$ is a sum of two squares in $K$, then there exist nontrivial cohomological invariants 
\begin{align*}
h_6: R(*) &\to H^6(*,\ZZ/2\ZZ),\\
h_8: R(*) &\to H^8(*,\ZZ/2\ZZ)
\end{align*}
such that if $A = N_{E/F}(C)$, then
\begin{align*}
h_6(K(A,-,\gamma)) &= e_6(N_{E/F}(n)),\\
h_8(K(A,-,\gamma)) &= (-\gamma_1\gamma_2^{-1}) \cup (-\gamma_2\gamma_3^{-1}) \cup e_6(N_{E/F}(n)).
\end{align*}
\end{corollary}

\section{Isotropy of Tits construction}

In this section we continue to assume that the base field $K$ is of characteristic not $2$ or $3$.

\subsection{$E_8/P_8$ as a compactification of $D_8/A_7\rtimes\ZZ/2\ZZ$}

Consider the split group of type $E_8$ over $K$ and the action of its subgroup of type $D_8$ on the projective homogeneous variety $E_8/P_8$. Since $D_8$ being the fixed point subgroup of an involution $\sigma$ is known to be spherical, the number of orbits (in the geometrical sense) must be finite, in particular, there is an open orbit. Using \cite[Lemma~2.9 and 2.11]{springer2011decompositions} we can give a precise description of this orbit: it consists of parabolic subgroups $P$ of type $P_8$ such that $\sigma(P)$ is opposite to $P$, that is $P\cap\sigma(P)=L$ is a Levi subgroup of $P$. In particular, $L$ is stable under the action of $\sigma$, a fortiori its commutator subgroup (of type $E_7$) and the centralizer of the commutator subgroup (of type $A_1$) are stable under the action of $\sigma$. In particular, the subgroup of type $E_7+A_1$ generated by them is also stable under the action of $\sigma$, and the stabilizer of a point in the open orbit is the fixed points subgroup of $\sigma$ acting on $E_7+A_1$. Using \cite[Table~I]{kollross2009exceptional} we see that this subgroup is of type $A_7$; more precisely, it is known to be $\SL_8/\mu_2\rtimes\ZZ/2\ZZ$.

\begin{lemma}\label{lem:homog}
Let $[\xi]$ be in $H^1(K,G)$, $H$ be a closed subgroup of $G$. Then ${}_\xi(G/H)$ has a $K$-rational point if and only if $[\xi]$ comes from some $[\zeta]\in H^1(K,H)$.
\end{lemma}
\begin{proof}
See \cite[Proposition~37]{serre}.
\end{proof}

\begin{lemma}\label{lem:hyp}
Let $[\xi]\in H^1(K,\PGO_{2n})$ be in the image of $H^1(K,\GL_n/\mu_2\rtimes\ZZ/2\ZZ)$. Then there exists a quadratic field extension $E/K$ such that the orthogonal involution corresponding to $\xi_E$ is hyperbolic.
\end{lemma}
\begin{proof}
Consider the following short exact sequence:
$$
H^1(K,\GL_n/\mu_2)\to H^1(K,\GL_n/\mu_2\rtimes\ZZ/2\ZZ)\to H^1(K,\ZZ/2\ZZ),
$$
and take $E/K$ corresponding to the image in $H^1(K,\ZZ/2\ZZ)$ of $[\zeta]$ in $H^1(K,\GL_n/\mu_2\rtimes\ZZ/2\ZZ)$ whose image in $H^1(K,\PGO_{2n})$ is $[\xi]$. Passing to $E$ we see that $[\xi_E]$ comes from $H^1(E,\GL_n/\mu_2)$ and so produces a hyperbolic involution.
\end{proof}

\begin{theorem} \label{tits-index}
Let $K$ be a $2$-special (that is with no odd degree extensions) field of characteristic not $2$ and $3$, $L$ be a Lie algebra of type $E_8$ obtained via the Tits construction. Then the group corresponding to $L$ is not of Tits index $E_{8,1}^{133}$.
\end{theorem}
\begin{proof}
Assume the contrary. Obviously the base field is infinite, for there are only split groups of type $E_8$ over finite fields. Let $L$ be obtained via the Tits construction from $C_1$ and $\mathcal{H}_3(C_2,\gamma)$ for some octonion algebras $C_1$ and $C_2$. Denote by $[\zeta]$ in $H^1(K,E_8)$ the class corresponding to $L$. By Proposition~\ref{prop:D8} $L$ contains a Lie subalgebra of type $D_8$, namely $\mathfrak{so}(\langle \gamma_i\rangle n_1 \perp \langle -\gamma_j^{-1} \rangle n_2)$, and so the corresponding group contains a subgroup of type $D_8$ (see \cite[Expos\'e~XXII, Corollaire~5.3.4]{demazure}). This means that ${}_\zeta(E_8/D_8)$ contains a $K$-rational point, and by Lemma~\ref{lem:homog} $[\zeta]$ comes from some $[\xi]\in H^1(K,D_8)$.

Now ${}_\xi(E_8/P_8)$ is a smooth compactification of its open subvariety $U={}_\zeta(D_8/A_7\rtimes\ZZ/2\ZZ)$ and by the assumption has a rational point. This means that there is a parabolic subgroup of type $P_8$ inside ${}_\xi E_8$, and the unipotent radical of an opposite parabolic subgroup defines an open subvariety in ${}_\xi(E_8/P_8)$ isomorphic to ${\mathbb A}^{57}$. Since the base field is infinite, there is a rational point in ${\mathbb A}^{57}\cap U$. Applying Lemma~\ref{lem:homog} and Lemma~\ref{lem:hyp} we see that the quadratic form $\langle \gamma_i\rangle n_1 \perp \langle -\gamma_j^{-1} \rangle n_2$ becomes hyperbolic over a quadratic field extension $E/K$. It follows that $e_3(n_1)+e_3(n_2)$ is trivial over $E$, hence $n_1-n_2$ belongs to $I^4$ and so is hyperbolic over $E$. Now $n_1-n_2$ is divisible by the discriminant of $E$ and so $e_3(n_1)+e_3(n_2)$ is a sum of two symbols with a common slot. But the Rost invariant of the anisotropic kernel of type $E_7$ is $e_3(n_1)+e_3(n_2)$, and applying \cite[Theorem~10.18]{garibaldi2015shells} we see that this group must be isotropic, a contradiction.
\end{proof}

Note that \cite[Appendix~A]{garibaldi2009cohomological} provides an example of a strongly inner group of type $E_7$ over a $2$-special field, hence an example of a group of Tits index $E_{8,1}^{133}$ over such a field.

\begin{corollary}
     Suppose $K$ is a field such that $-1$ is a sum of two squares, and let $L$ be a Lie algebra over $K$ of type $E_8$ obtained via the Tits construction.
     \begin{enumerate}[\rm (i)]
     \item If $h_8(L)\ne 0$ then $L$ is anisotropic.
     \item If $-1$ is a square in $K$ and $h_6(L) \ne 0$   then $L$ has $K$-rank $\le 1$.
     \end{enumerate}
\end{corollary}

\begin{proof}
    (i) Suppose $L$ is isotropic. After an odd-degree extension $F/K$, we can assume that $L_F$ does not have Tits index $E_{8,1}^{133}$ by Theorem \ref{tits-index}. We can also assume that $L_F$ does not have Tits index $E_{8,2}^{78}$ because then its anisotropic kernel would be of type ${}^1E_{6,0}^{78}$ and every such group becomes isotropic over an odd-degree extension (see \cite[Exercise 22.9]{garibaldi2003cohomological}). In each of the remaining possible indices from \cite[p.\! 60]{tits1966classification}, $L_F$ corresponds to a class in the image of $H^1(F,\operatorname{Spin}_{14}) \to H^1(F,E_8)$, which means it is isomorphic to $K(A,-) \simeq K(A,-,(1,-1,1))$ for some bi-octonion algebra $A$. Then clearly $h_8(L_F) = 0$, so $h_8(L) = 0$.
    
    (ii) Suppose $L$ has $K$-rank $\ge 2$. Then there is an odd-degree extension $F/K$ such that $L_F$ corresponds to a class in the image of $H^1(F,\operatorname{Spin}_{12}) \to H^1(F,E_8)$. Its anisotropic kernel is a subgroup of $\Spin(q)$ for some 12-dimensional form $q$ belonging to $I^3(K)$, and by a well-known theorem of Pfister (see \cite[Theorem 17.13]{garibaldi2009cohomological}) $q$ is similar to $n_1 - n_2$ for a pair of 3-Pfister forms $n_i$ with a common slot, say $n_i = \llangle x, y_i,z_i \rrangle$. If $C_i$ is the octonion algebra corresponding to $n_i$ then we have $L_F \simeq K(C_1 \otimes C_2,-)$, and since $-1$ is a square, \[h_6(L) = e_6(\llangle x,y_1,z_1, x, y_2,z_2\rrangle) =  (-1) \cup (x) \cup (y_1)\cup (y_2) \cup (z_1)\cup (z_2) = 0. \qedhere \] 
\end{proof}

\newpage

\bibliographystyle{acm}
\bibliography{mybib}

\end{document}